\documentclass{amsart}

\usepackage{ adjustbox }
\usepackage{ amsthm }
\usepackage{ amssymb }
\usepackage{ amsmath }
\usepackage{ array }
\usepackage{ booktabs }
\usepackage{ breqn }
\usepackage{ cite }
\usepackage{ color }
\usepackage{ enumitem }
\usepackage{ latexsym }
\usepackage{ url }
\usepackage[all]{ xy }

\theoremstyle{definition}

\newtheorem{definition}{Definition}[section]

\newtheorem{conjecture}[definition]{Conjecture}
\newtheorem{example}[definition]{Example}

\newtheorem{remark}[definition]{Remark}
\theoremstyle{plain}

\newtheorem{lemma}[definition]{Lemma}

\newtheorem{theorem}[definition]{Theorem}

\newcommand{\rational}{\mathbb{Q}}
\allowdisplaybreaks

\begin{document}

\title{Special Identities for Comtrans Algebras}

\author{Murray R. Bremner}

\address{Department of Mathematics and Statistics, University of Saskatchewan, Saskatoon, Canada}

\email{bremner@math.usask.ca}

\author{Hader A. Elgendy}

\address{Faculty of Science, Department of Mathematics, Damietta
University, Damietta 34517, Egypt}

\email{haderelgendy42@hotmail.com}



\thanks{The research of the first author was supported by the Discovery Grant 
\emph{Algebraic Operads} from NSERC, 
the Natural Sciences and Engineering Research Council of Canada.
He thanks the Department of Mathematics at Damietta University in Egypt
for its hospitality during May 2018.}

\begin{abstract}
Comtrans algebras, arising in web geometry, have two trilinear operations,
commutator and translator.
We determine a Gr\"obner basis for the comtrans operad, 
and state a conjecture on its dimension formula.
We study multilinear polynomial identities for the special commutator 
$[x,y,z] = xyz-yxz$ 
and special translator 
$\langle x, y, z \rangle = xyz-yzx$ 
in associative triple systems.
In degree 3, the defining identities for comtrans algebras generate all identities.
In degree 5, we simplify known identities for each operation
and determine new identities relating the operations.
In degree 7, we use representation theory of the symmetric group 
to show that each operation satisfies identities which do not follow from 
those of lower degree but there are no new identities relating the operations.
We use noncommutative Gr\"obner bases to construct the universal associative envelope
for the special comtrans algebra of $2 \times 2$ matrices.
\end{abstract}

\maketitle


\section{Introduction}

Comtrans algebras were introduced by Smith \cite[\S3]{Smith1988} to answer
a problem in web geometry \cite{AG,Chern} posed by Goldberg \cite[Problem X.3.9]{CPS1990}:
to find the algebraic structure on the tangent bundle of the coordinate ternary loop of a 4-web \cite[\S3.7]{G1988}.
Comtrans algebras are a common ternary generalization of Lie algebras,
Malcev algebras \cite{K} and Akivis algebras \cite{A}:
every such algebra can be given the structure of a comtrans algebra
\cite[\S5]{RS1996}.
A generalization of Lie's Third Fundamental Theorem connects formal ternary loops with
comtrans algebras \cite[\S5]{Smith1988}.
For physical applications of comtrans algebras, see \cite[p.~321]{Smith1993}:
``\dots the Lorentz metric on 4-dimensional real space-time provides a simple comtrans algebra
that extends the 3-dimensional vector triple product comtrans algebra.''

\begin{definition}
Smith \cite[\S3]{Smith1988}.
A \emph{comtrans algebra} is a vector space $A$ with two trilinear operations
$A \times A \times A \to A$,
the \emph{commutator} $[x,y,z]$ and the \emph{translator} $\langle x,y,z \rangle$,
satisfying the following polynomial identities for all $x, y, z \in A$:
\begin{align}
\label{alternating}
&
[x,y,z] + [y,x,z] = 0,
\\
\label{Jacobi}
&
\langle x,y,z \rangle + \langle y,z,x \rangle + \langle z,x,y \rangle = 0,
\\
\label{comtrans}
&
[x,y,z] + [z,y,x] = \langle x,y,z \rangle + \langle z,y,x \rangle.
\end{align}
The commutator alternates in the first two arguments \eqref{alternating},
the translator satisfies the Jacobi identity \eqref{Jacobi}, and
the operations are related by the comtrans identity \eqref{comtrans}.
\end{definition}

\begin{example}
If $T$ is a Lie triple system with bracket $[x,y,z]$ then letting both commutator and translator
equal $[x,y,z]$ gives $T$ the structure of a comtrans algebra $T^{CT}$.
If $T$ is obtained from a Lie algebra $L$ by $[x,y,z] = [[x,y],z]$ then
Shen \& Smith \cite[Theorem 3.2]{SS1993} have shown that $L$ is simple if and only if $T^{CT}$
is simple.
\end{example}

\begin{example}
Let $A_{m,n}$ denote the vector space of $m \times n$ matrices over $\mathbb{F}$.
Fix matrices $p$ ($n \times n$) and $q$ ($m \times m$) over $\mathbb{F}$.
Define a trilinear operation on $A_{m,n}$ by $(x,y,z) = x p y^t q z$.
Setting $[x,y,z] = (x,y,z) - (y,x,z)$ and $\langle x,y,z \rangle = (x,y,z) - (y,z,x)$
gives $A_{m,n}$ the structure of a comtrans algebra \cite[Example 2.1]{SS1993}.
\end{example}

\begin{definition}
\label{ATS}
An \emph{associative triple system} \cite{Lister} is a vector space $A$ with a trilinear operation
$xyz$ satisfying $(vw(xyz)) = (v(wxy)z) = (vw(xyz))$.
The \emph{special commutator} and \emph{special translator} in $A$ are
$[x,y,z] = xyz - yxz$ and $\langle x,y,z \rangle = xyz - yzx$.
\end{definition}

\begin{remark}
In the terminology of \cite{BPCTO}, two trilinear operations in $\mathbb{Q}S_3$ are
equivalent if they generate the same left ideal.
The special commutator is equivalent to the $q = 2$ case
of the $q$-deformed anti-Jordan triple product.
The special translator is the same as the cyclic commutator, and is equivalent to
the operation $2xyz - yzx - zxy$ which represents to the identity matrix $I_2$ in
the simple 2-sided ideal corresponding to partition $2{+}1$.
\end{remark}

We restrict consideration to multilinear polynomial identities, since this allows us 
to apply the representation theory of the symmetric group.
This approach to polynomial identities was introduced 
by Malcev \cite{Malcev} and Specht \cite{Specht}; 
its computer implementation was pioneered by Hentzel \cite{H1,H2}.
This point of view also fits very naturally into the theory of algebraic operads \cite{BD,LV,MSS}.
For an introduction to operads which emphasizes connections to homotopical algebra,
see Vallette \cite{Vallette}.

In \S\ref{prelim} we recall basic definitions on triple systems.
In \S\ref{operadsection} we show how algebraic operads may be used to discuss polynomial identities.
In particular, we calculate a Gr\"obner basis for the comtrans operad, and 
explain how Gr\"obner bases of operads may be used to compute polynomial identities.
In \S\ref{degree3section} we show that every identity in degree 3 for the commutator and translator 
follows from \eqref{alternating}--\eqref{comtrans}.
In \S\ref{degree5section} we use computer algebra to find explicit generators for the $S_5$-module
of identities which do not follow from those of degree 3.
We consider three cases: the commutator by itself, the translator by itself, 
and identities in which each term contains both operations.
In \S\ref{degree7section} we use a constructive version of the representation theory of the
symmetric group to demonstrate that there are new identities in degree 7 for each operation
separately but no new identities relating the operations.
In \S\ref{envelopesection} we construct the universal associative enveloping algebra of 
the special comtrans algebra $M^{CT}$ obtained from the associative triple system $M$ of $2 \times 2$ matrices.

Our results are valid over any field of characteristic 0.
Our computations were performed using Maple worksheets written by the authors.


\section{Preliminaries}
\label{prelim}

\begin{lemma}
In an associative triple system $A$, the special commutator and special translator
satisfy the relations \eqref{alternating}--\eqref{comtrans}.
\end{lemma}

\begin{proof}
Trivial calculation.
This result appears in a more general setting and without any assumption of associativity 
in \cite[(3.4)-(3.6)]{Smith1988}.
\end{proof}

\begin{definition}
Let $M$ be an associative triple system and let $M^{CT}$ be the comtrans algebra defined on 
the underlying vector space by the special commutator and special translator.
We say that a comtrans algebra $A$ is \emph{special} if there exists an associative triple system
$M$ and an injective morphism $A \to M^{CT}$ of comtrans algebras; otherwise, $A$ is \emph{exceptional}.
(This terminology is motivated by the definitions of special and exceptional Jordan algebras \cite{MC}.)
\end{definition}

\begin{remark}
Tercom (ternary commutator) algebras were introduced by Rossmanith \& Smith \cite{RS1996}.
They have trilinear operations $\lambda(x,y,z)$ and $\rho(x,y,z)$ which satisfy the following
polynomial identities:
\begin{align*}
&
\lambda(x,y,z) + \lambda(y,x,z) = 0,
\\
&
\rho(x,y,z) + \rho(x,z,y) = 0,
\\
&
\lambda(x,y,z) +
\lambda(y,z,x) +
\lambda(z,x,y)
=
\rho(x,y,z) +
\rho(y,z,x) +
\rho(z,x,y).
\end{align*}
These identities hold for the left and right commutators
$\lambda(x,y,z) = xyz - yxz$ and $\rho(x,y,z) = xyz - xzy$ in every associative triple system.
Tercom algebras are equivalent to comtrans algebras in the sense that
\begin{align*}
&
[x,y,z] = \lambda(x,y,z),
\\
&
\langle x,y,z \rangle = \lambda(x,y,z) + \rho(y,x,z),
\\
&
\rho(x,y,z) = \langle y,x,z \rangle - [y,x,z].
\end{align*}
\end{remark}

\begin{remark}
\label{wacremark}
The weakly anticommutative operation \cite{BPCTO}, which is equivalent to the operation
$\{x,y,z\} = xyz + xzy - 2 zyx$,
satisfies the symmetric sum identity in every associative triple system:
\[
\sum_{\sigma \in S_3} \{ x^\sigma, y^\sigma, z^\sigma \} = 0.
\]
This allows us to define special comtrans algebras in terms of a single operation:
\begin{align*}
[x,y,z] &= \tfrac12 \big( \{z,x,y\} - \{z,y,x\} \big),
\\
\langle x,y,z \rangle &= \tfrac16 \big( 4\{x,z,y\} + 2\{y,z,x\} + \{z,x,y\} - \{z,y,x\} \big),
\\
\{x,y,z\} &= [y,z,x] + \langle x,y,z \rangle + \langle x,z,y \rangle.
\end{align*}
\end{remark}

\begin{remark}
Every special comtrans algebra becomes a special anti-Jordan triple system
using the trilinear operation $(x,y,z) = [x,y,z] - \langle z,y,x \rangle = xyz - zyx$.
\end{remark}


\section{Algebraic operads}
\label{operadsection}

We consider operads in the symmetric monoidal category of vector spaces over
a field of characteristic 0;
equivalently, $\mathbb{Z}$-graded vector spaces concentrated in degree 0.
We say \emph{degree} instead of \emph{arity} since our motivation comes
from nonassociative algebra and we never refer to the homological degree
(that is, all operations have homological degree 0).


\subsection{Basic definitions}

\newcommand{\mygamma}{\raisebox{0.0pt}{{\raisebox{+0.0pt}{\scalebox{1.0}{$\circ$}}}}}
\newcommand{\mydelta}{\raisebox{0.0pt}{{\raisebox{-0.0pt}{\scalebox{1.0}{$\bullet$}}}}}

A monomial of weight $w$ (and hence degree $d(w) = 2w{+}1$) in two trilinear operations 
$\gamma$ and $\delta$ is a sequence of $d(w)$ distinct arguments (usually identified with 
a permutation of the variables $x_1, \dots, x_{d(w)}$) into which $w$ operation symbols 
have been inserted (each symbol may be either $\gamma$ or $\delta$).
There is a bijection between monomials of weight $w$ and complete ternary trees with $w$ 
internal nodes each labelled by one of the operation symbols, and $d(w)$ leaves labelled 
by distinct arguments.
For example, writing $\circ$ for $\gamma$ and $\bullet$ for $\delta$,
\[
\gamma( x_3, \delta( x_1, x_4, x_5 ), \gamma( x_9, \delta( x_6, x_8, x_7 ), x_2 ) )
\quad\longleftrightarrow\quad
\adjustbox{valign=m}{
\begin{xy}
(   0,   0 )*{\mygamma} = "root";
( -15,  -6 )*{x_3} = "l";
(   0,  -6 )*{\mydelta} = "m";
(  15,  -6 )*{\mygamma} = "r";
{ \ar@{-} "root"; "l" };
{ \ar@{-} "root"; "m" };
{ \ar@{-} "root"; "r" };
( -5, -12 )*{\raisebox{-5.0pt}{$x_1$}} = "ml";
(  0, -12 )*{x_4} = "mm";
(  5, -12 )*{x_5} = "mr";
{ \ar@{-} "m"; "ml" };
{ \ar@{-} "m"; "mm" };
{ \ar@{-} "m"; "mr" };
( 10, -12 )*{x_9} = "rl";
( 15, -12 )*{\mydelta} = "rm";
( 20, -12 )*{x_2} = "rr";
{ \ar@{-} "r"; "rr" };
{ \ar@{-} "r"; "rm" };
{ \ar@{-} "r"; "rl" };
( 10, -18 )*{x_6} = "rml";
( 15, -18 )*{x_8} = "rmm";
( 20, -18 )*{x_7} = "rmr";
{ \ar@{-} "rm"; "rml" };
{ \ar@{-} "rm"; "rmm" };
{ \ar@{-} "rm"; "rmr" };
\end{xy}
}
\]

\smallskip

\begin{definition}
\label{defoperads}
Let $\mathcal{T}$ be the free weight-graded symmetric operad generated by ternary operations
$\gamma$, $\delta$ with no symmetry, 
and let $\mathcal{T}(w)$ be the homogeneous subspace of weight $w$.
Let $\mathcal{A}$ be the weight-graded symmetric operad generated by a ternary operation
$\alpha$ with no symmetry satisfying ternary associativity (Definition \ref{ATS}), 
and let $\mathcal{A}(w)$ be the homogeneous subspace of weight $w$.
Using the formula for ternary Catalan numbers, we have
\[
\dim\mathcal{T}(w) = \frac{2^w}{2w{+}1} \binom{3w}{w} (2w{+}1)!,
\qquad\qquad
\dim\mathcal{A}(w) = (2w{+}1)!.
\]
Since we use the weight grading, $\mathcal{T}(w)$ and $\mathcal{A}(w)$ are modules over $S_{2w+1}$.
\end{definition}

\begin{definition}
\label{expansionmap}
Since $\mathcal{T}$ is free, a morphism with domain $\mathcal{T}$ is determined
by its values on $\gamma$ and $\delta$.
We define the \emph{expansion map} $X\colon \mathcal{T} \to \mathcal{A}$ by
\[
X( \gamma ) = \alpha - (12) \cdot \alpha,
\quad
X( \delta ) = \alpha - (132) \cdot \alpha,
\quad
p \cdot x_1 x_2 x_3 = x_{p^{-1}(1)} x_{p^{-1}(2)} x_{p^{-1}(3)}.
\]
Permutations act on positions of arguments (not subscripts).
Thus
\[
X(\gamma)(x,y,z) = \alpha(x,y,z) - \alpha(y,x,z),
\qquad
X(\delta)(x,y,z) = \alpha(x,y,z) - \alpha(y,z,x).
\]
We use the more convenient notation
\[
\gamma(x,y,z) = [x,y,z],
\qquad
\delta(x,y,z) = \langle x,y,z \rangle,
\qquad
\alpha(x,y,z) = xyz.
\]
Thus $X$ may be written in terms of the special ternary commutator and translator:
\[
[x,y,z] \;\xrightarrow{\quad X \quad}\; xyz - yxz,
\qquad\qquad
\langle x,y,z \rangle \;\xrightarrow{\quad X \quad}\; xyz - yzx.
\]
We write $X_w$ for the restriction of $X$ to $\mathcal{T}(w)$.
\end{definition}

\begin{lemma}
We have $X_w( \mathcal{T}(w) ) \subseteq \mathcal{A}(w)$ for $w \ge 0$.
The kernel of $X_w$ is the $S_{2w+1}$-submodule of $\mathcal{T}(w)$ of
multilinear polynomial identities of degree $2w{+}1$ satisfied by
the special commutator and translator in every associative triple system.
\end{lemma}


\subsection{A Gr\"obner basis for the comtrans operad}
\label{GBCO}

The notion of Gr\"obner basis is not well-defined in general for ideals in symmetric operads;
see for example \cite[Proposition 5.2.2.5]{BD}.
In order to apply the algorithm for computing a Gr\"obner basis to 
a symmetric operad $\mathcal{P}$, 
we must first apply the forgetful functor from symmetric operads to shuffle operads;
equivalently, we must ignore the symmetric group actions%
\footnote{We thank Vladimir Dotsenko for emphasizing the importance of the forgetful functor 
from symmetric operads to shuffle operads in this subsection and the next.}.
We then compute a Gr\"obner basis of the corresponding ideal in the shuffle operad 
$\mathcal{P}_{\text{sh}}$; see for example \cite[Example 5.3.4.3]{BD}.
Shuffle operads and their Gr\"obner bases were introduced by Dotsenko \& Khoroshkin \cite{DK};
for a more detailed exposition, see \cite[Chapter 5]{BD}).

In the free weight-graded symmetric operad $\mathcal{T}$ of Definition \ref{defoperads}, 
the homogeneous subspace $\mathcal{T}(1)$ has the following ordered basis of ternary operations:
\begin{equation}
\label{basis12}
\begin{array}{cccccc}
\langle x, y, z \rangle, &
\langle x, z, y \rangle, &
\langle y, x, z \rangle, &
\langle y, z, x \rangle, &
\langle z, x, y \rangle, &
\langle z, y, x \rangle, 
\\ {}
[ x, y, z ], &
[ x, z, y ], &
[ y, x, z ], &
[ y, z, x ], &
[ z, x, y ], &
[ z, y, x ].
\end{array}
\end{equation}
As a symmetric operad, $\mathcal{T}(1)$ has the natural $S_3$-module structure
where $S_3$ acts by permuting the variables $x, y, z$.
The free symmetric operad $\mathcal{T}$ is generated by this 12-dimensional $S_3$-module
in degree 3 (weight 1).
Note also that $\mathcal{T}(1)$ is generated as an $S_3$-module by the two operations
$\langle x, y, z \rangle$ and $[ x, y, z ]$.

We now apply the forgetful functor to $\mathcal{T}$ to obtain the free shuffle operad
$\mathcal{T}_{\text{sh}}$.
The homogeneous subspace $\mathcal{T}_{\text{sh}}(1)$ has the same basis \eqref{basis12},
but we have forgotten the action of $S_3$: 
that is, $\mathcal{T}_{\text{sh}}(1)$ is not an $S_3$-module, it is simply a vector space.
Hence the free shuffle operad $\mathcal{T}_{\text{sh}}$ is generated by the same 
12-dimensional space, but now we must regard each of the basis monomials \eqref{basis12}
as a distinct operation.
For the most general statement of this isomorphism, see \cite[Corollary 5.3.3.3]{BD}.

Let $\mathcal{I} \subset \mathcal{T}$ be the (weight-graded) ideal in the free symmetric operad 
$\mathcal{T}$ such that the quotient operad $\mathcal{CT} = \mathcal{T} / \mathcal{I}$ is 
the symmetric comtrans operad.
The ideal $\mathcal{I}$ is generated by its homogeneous component $\mathcal{I}(1)$,
which in turn is generated as an $S_3$-module by relations 
\eqref{alternating}--\eqref{comtrans} defining comtrans algebras. 
In order to determine the homogeneous component $\mathcal{I}_{\text{sh}}(1)$ of the ideal
$\mathcal{I}_{\text{sh}}$ in the corresponding shuffle operad $\mathcal{T}_{\text{sh}}$,
we must first find a linear basis of $\mathcal{I}(1)$.

\begin{figure}[ht]
\small
$\left[
\begin{array}{r@{\,}r@{\,}r@{\,}r@{\,}r@{\,}r|r@{\;\;}r@{\;\;}r@{\;\;}r@{\;\;}r@{\;\;}r}
\cdot & \cdot & \cdot & \cdot & \cdot & \cdot & 1 & \cdot & 1 & \cdot & \cdot & \cdot \\[-1mm]
\cdot & \cdot & \cdot & \cdot & \cdot & \cdot & \cdot & 1 & \cdot & \cdot & 1 & \cdot \\[-1mm]
\cdot & \cdot & \cdot & \cdot & \cdot & \cdot & 1 & \cdot & 1 & \cdot & \cdot & \cdot \\[-1mm]
\cdot & \cdot & \cdot & \cdot & \cdot & \cdot & \cdot & \cdot & \cdot & 1 & \cdot & 1 \\[-1mm]
\cdot & \cdot & \cdot & \cdot & \cdot & \cdot & \cdot & 1 & \cdot & \cdot & 1 & \cdot \\[-1mm]
\cdot & \cdot & \cdot & \cdot & \cdot & \cdot & \cdot & \cdot & \cdot & 1 & \cdot & 1 \\
\midrule
1 & \cdot & \cdot & 1 & 1 & \cdot & \cdot & \cdot & \cdot & \cdot & \cdot & \cdot \\[-1mm]
\cdot & 1 & 1 & \cdot & \cdot & 1 & \cdot & \cdot & \cdot & \cdot & \cdot & \cdot \\[-1mm]
\cdot & 1 & 1 & \cdot & \cdot & 1 & \cdot & \cdot & \cdot & \cdot & \cdot & \cdot \\[-1mm]
1 & \cdot & \cdot & 1 & 1 & \cdot & \cdot & \cdot & \cdot & \cdot & \cdot & \cdot \\[-1mm]
1 & \cdot & \cdot & 1 & 1 & \cdot & \cdot & \cdot & \cdot & \cdot & \cdot & \cdot \\[-1mm]
\cdot & 1 & 1 & \cdot & \cdot & 1 & \cdot & \cdot & \cdot & \cdot & \cdot & \cdot \\
\midrule
-1 & \cdot & \cdot & \cdot & \cdot & -1 & 1 & \cdot & \cdot & \cdot & \cdot & 1 \\[-1mm]
\cdot & -1 & \cdot & -1 & \cdot & \cdot & \cdot & 1 & \cdot & 1 & \cdot & \cdot \\[-1mm]
\cdot & \cdot & -1 & \cdot & -1 & \cdot & \cdot & \cdot & 1 & \cdot & 1 & \cdot \\[-1mm]
\cdot & -1 & \cdot & -1 & \cdot & \cdot & \cdot & 1 & \cdot & 1 & \cdot & \cdot \\[-1mm]
\cdot & \cdot & -1 & \cdot & -1 & \cdot & \cdot & \cdot & 1 & \cdot & 1 & \cdot \\[-1mm]
-1 & \cdot & \cdot & \cdot & \cdot & -1 & 1 & \cdot & \cdot & \cdot & \cdot & 1
\end{array}
\right]
\xrightarrow{\;\mathrm{RCF}\;}
\left[
\begin{array}{r@{\;\;}r@{\;\;}r@{\;\;}r@{\;\;}r@{\,}r|r@{\;\;}r@{\,}r@{\,}r@{\,}r@{\,}r}
1 & \cdot & \cdot & \cdot & \cdot & 1 & \cdot & \cdot & \cdot & 1 & 1 & -1 \\[-1mm]
\cdot & 1 & \cdot & \cdot & \cdot & -1 & \cdot & \cdot & 1 & -1 & \cdot & 1 \\[-1mm]
\cdot & \cdot & 1 & \cdot & \cdot & -1 & \cdot & \cdot & \cdot & -1 & -1 & 1 \\[-1mm]
\cdot & \cdot & \cdot & 1 & \cdot & 1 & \cdot & \cdot & \cdot & \cdot & \cdot & \cdot \\[-1mm]
\cdot & \cdot & \cdot & \cdot & 1 & 1 & \cdot & \cdot & -1 & 1 & \cdot & -1 \\[-1mm]
\cdot & \cdot & \cdot & \cdot & \cdot & \cdot & 1 & \cdot & \cdot & 1 & 1 & \cdot \\[-1mm]
\cdot & \cdot & \cdot & \cdot & \cdot & \cdot & \cdot & 1 & 1 & \cdot & \cdot & 1
\end{array}
\right]$
\vspace{-2mm}
\caption{Computation of a Gr\"obner basis for the comtrans operad}
\label{18x12matrix}
\end{figure}

We form the $18 \times 12$ matrix (Figure \ref{18x12matrix}, left) 
whose three $6 \times 12$ blocks contain the coefficient vectors of all permutations of 
relations \eqref{alternating}--\eqref{comtrans} with respect to the ordered basis \eqref{basis12}.
The row space of this relation matrix may be identified with the $S_3$-module $\mathcal{I}(1)$.
The homogeneous subspace $\mathcal{I}_{\text{sh}}(1)$ of the shuffle ideal $\mathcal{I}_{\text{sh}}$
is the same row space, regarded simply as a vector space, forgetting the action of $S_3$.

Relations \eqref{alternating}--\eqref{comtrans} are linear (weight 1) relations between 
ternary operations.
For linear relations, the computation of a Gr\"obner basis for the shuffle ideal
$\mathcal{I}_{\text{sh}}$ reduces to applying Gaussian elimination to the relation matrix.
We compute the RCF (row canonical form) of the relation matrix (Figure \ref{18x12matrix}, right).
The nonzero rows of the RCF are (the coefficient vectors of) the Gr\"obner basis for 
$\mathcal{I}_{\text{sh}}$.
These rows represent the following relations among the 12 operations \eqref{basis12} which 
generate the free shuffle operad $\mathcal{T}_{\text{sh}}$, 
and this is the required Gr\"obner basis:
\begin{align*}
&
  \langle x, y, z \rangle
+ \langle z, y, x \rangle
+ [ y, x, z ]
- [ z, y, x ],
\\[-1mm]
&
  \langle x, z, y \rangle
- \langle z, x, y \rangle
+ \langle z, y, x \rangle
+ [ y, x, z ]
+ [ z, x, y ],
\\[-1mm]
&
  \langle y, x, z \rangle
+ \langle z, x, y \rangle
- [ y, x, z ]
- [ z, x, y ],
\\[-1mm]
&
  \langle y, z, x \rangle
+ \langle z, x, y \rangle
- \langle z, y, x \rangle
- [ y, x, z ]
+ [ z, y, x ],
\\[-1mm]
&
  [ x, y, z ]
+ [ y, x, z ],
\\[-1mm]
&
  [ x, z, y ]
+ [ z, x, y ],
\\[-1mm]
&
  [ y, z, x ]
+ [ z, y, x ].
\end{align*}
Equivalently, in terms of rewrite rules, we have:
\begin{equation}
\label{rewriterules}
\begin{array}{ll}
\langle x,y,z \rangle \; &\longrightarrow \;\;\;
- \langle z,y,x \rangle - [y,x,z] + [z,y,x],
\\
\langle x,z,y \rangle \; &\longrightarrow \;\;\;
 \langle z,x,y \rangle - \langle z,y,x \rangle - [y,x,z] - [z,x,y],
\\
\langle y,x,z \rangle \; &\longrightarrow \;\;\;
- \langle z,x,y \rangle + [y,x,z] + [z,x,y],
\\
\langle y,z,x \rangle \; &\longrightarrow \;\;\;
- \langle z,x,y \rangle + \langle z,y,x \rangle + [y,x,z] - [z,y,x],
\\ {}
[x,y,z] \; &\longrightarrow \;\;\; - [y,x,z],
\\ {}
[x,z,y] \; &\longrightarrow \;\;\; - [z,x,y],
\\ {}
[y,z,x] \; &\longrightarrow \;\;\; - [z,y,x].
\end{array}
\end{equation}
If $x$, $y$, $z$ are multilinear monomials of arbitrary degree, 
and $x \prec y \prec z$ in the deglex extension of the operation order,
then the rewrite rules show how to move greater factors to the left,
to the extent allowed by the comtrans relations.

\begin{conjecture}
If $\mathcal{CT}$ is the symmetric weight-graded comtrans operad then
\[
\dim \mathcal{CT}(w) = \frac{ (3w)! }{ w! (3!)^w } \cdot 5^w \qquad (w \ge 0).
\]
Using terminology from nonassociative algebra, 
this is the dimension of the multilinear subspace of degree $2w{+}1$
in the free comtrans algebra on $2w{+}1$ generators.
\end{conjecture}

We implemented the rewrite rules in Maple to verify the first four terms 1, 5, 250, 35000
by computing all multilinear normal forms in degrees $1, 3, 5, 7$.
The factor $5^w$ comes from the number of basis operations \eqref{basis12} which are
irreducible with respect to the rewrite rules \eqref{rewriterules}.
The other factor is OEIS sequence A025035 which counts
(i)
set partitions of $\{ 1, 2, \dots, 3w \}$ into parts of size 3,
(ii)
rooted phylogenetic (non-planar) complete ternary trees with $w$ internal vertices, and
(iii)
distinct multilinear monomials in degree $2w{+}1$ for a symmetric ternary operation: 
$( x^\sigma, y^\sigma, z^\sigma ) = ( x, y, z )$ for $\sigma \in S_3$.


\subsection{Polynomial identities and operadic Gr\"obner bases}
\label{PIOGB}

In principle, we can use Gr\"obner bases to determine polynomial identities satisfied by 
the special commutator and translator in every associative triple system.
We describe the method briefly to illustrate the connection between
polynomial identities and Gr\"obner bases of operads.
(We will not use this method;
instead, we apply computational techniques based on the LLL algorithm for lattice basis reduction
\cite{BPLLL} and the representation theory of the symmetric group \cite{BMP}.)

We express the expansion map (Definition \ref{expansionmap}) in terms of a single operad.
Let $\mathcal{P}$ be the free weight-graded symmetric operad generated by the ternary operations
$\alpha$, $\delta$, $\gamma$ with no symmetries.
Let $\mathcal{O} = \mathcal{P} / \mathcal{I}$ where $\mathcal{I}$ is generated by these relations:
\begin{align}
&
\gamma(x,y,z) + \gamma(y,x,z),
\label{C1}
\\[-1mm]
&
\delta(x,y,z) + \delta(y,z,x) + \delta(z,x,y),
\label{C2}
\\[-1mm]
&
\gamma(x,y,z) + \gamma(z,y,x) - \delta(x,y,z) - \delta(z,y,x),
\label{C3}
\\[-1mm]
&
\alpha(x,y,z) - \alpha(y,x,z) - \gamma(x,y,z),
\label{S1}
\\[-1mm]
&
\alpha(x,y,z) - \alpha(y,z,x) - \delta(x,y,z),
\label{S2}
\\[-1mm]
&
\alpha( v, \alpha(w,x,y), z ) - \alpha( \alpha(v,w,x), y, z ),
\label{A1}
\\[-1mm]
&
\alpha( v, w, \alpha(x,y,z) ) - \alpha( \alpha(v,w,x), y, z ).
\label{A2}
\end{align}
Relations \eqref{C1}-\eqref{C3} are the defining identities for comtrans algebras;
\eqref{S1}-\eqref{S2} express the comtrans operations $\gamma, \delta$ 
in terms of the associative operation $\alpha$;
and \eqref{A1}-\eqref{A2} are the defining identities for associative triple systems.
Relations \eqref{C1}-\eqref{S2} are linear (degree 3) and relations \eqref{A1}-\eqref{A2}
are quadratic (degree 5).
Thus $\mathcal{I}$ is generated by its homogeneous components 
$\mathcal{I}(1)$ and $\mathcal{I}(2)$.
A linear basis for each of these components is obtained by applying all permutations of
the variables and then applying Gaussian elimination to the resulting relation matrices.

We apply the forgetful functor from symmetric operads to shuffle operads as described 
in \S\ref{GBCO} and obtain the ideal $\mathcal{I}_{\text{sh}} \subset \mathcal{P}_{\text{sh}}$.
We must use an operation order such as $\alpha \succ \gamma \succ \delta$ to ensure that 
tree monomials containing $\alpha$ are greater than
tree monomials not containing $\alpha$.
This operation order extends to the revdeglex (reverse degree lexicographical) order on 
the tree monomials forming a linear basis of $\mathcal{P}_{\text{sh}}$.
We apply the algorithm for computing Gr\"obner bases of shuffle ideals to the relations
forming the linear bases of $\mathcal{I}_{\text{sh}}(1)$ and $\mathcal{I}_{\text{sh}}(2)$.
This computation produces relations of the form $f + g$ where $f$ consists of the terms 
containing $\alpha$ (and possibly also $\gamma$, $\delta$) and $g$ consists of the terms 
not containing $\alpha$.
If the Gr\"obner basis algorithm produces a polynomial $f + g$ with $f = 0$ then
$\mathcal{I}_{\text{sh}}$ contains a polynomial $g$ involving only $\gamma$, $\delta$
which is identically 0 in the quotient 
$\mathcal{O}_{\text{sh}} = \mathcal{P}_{\text{sh}} / \mathcal{I}_{\text{sh}}$.
Hence $g$ is a relation between $\gamma$, $\delta$ which belongs to the kernel of 
the expansion map, and so $g$ is a polynomial identity relating $\gamma$, $\delta$.


\section{Identities of degree 3}
\label{degree3section}

\begin{lemma}
Every multilinear identity satisfied by the special commutator and special translator
in every associative triple system follows from \eqref{alternating}--\eqref{comtrans}.
\end{lemma}

\begin{proof}
We follow \cite{BPLLL}.
For the expansion map $X_1 \colon \mathcal{T}(1) \to \mathcal{A}(1)$ in degree 3,
the permutations $xyz$, \dots, $zyx$ in lex order form a basis of $\mathcal{A}(1)$, and
the monomials $[x,y,z]$, \dots, $[z,y,x]$, $\langle x,y,z \rangle$, \dots, $\langle z,y,x \rangle$
form an ordered basis of $\mathcal{T}(1)$.
The matrix representing $X_1$ has rank 5 and nullity 7:
\[
[X_1] =
\left[
\begin{array}{rrrrrrrrrrrr}
1 & \cdot & -1 & \cdot & \cdot & \cdot & 1 & \cdot & \cdot & \cdot & -1 & \cdot \\[-1pt]
\cdot & 1 & \cdot & \cdot & -1 & \cdot & \cdot & 1 & -1 & \cdot & \cdot & \cdot \\[-1pt]
-1 & \cdot & 1 & \cdot & \cdot & \cdot & \cdot & \cdot & 1 & \cdot & \cdot & -1 \\[-1pt]
\cdot & \cdot & \cdot & 1 & \cdot & -1 & -1 & \cdot & \cdot & 1 & \cdot & \cdot \\[-1pt]
\cdot & -1 & \cdot & \cdot & 1 & \cdot & \cdot & \cdot & \cdot & -1 & 1 & \cdot \\[-1pt]
\cdot & \cdot & \cdot & -1 & \cdot & 1 & \cdot & -1 & \cdot & \cdot & \cdot & 1
\end{array}
\right]
\]
We compute the Hermite normal form $H$ of the transpose $[X_1]^t$
and an integer matrix $U$ such that $\det(U) = \pm 1$ and $U [X_1]^t = H$.
Rows 6-12 of $H$ are zero and so rows 6-12 of $U$ form a lattice basis $N_1$ 
for the integer nullspace of $[X_1]$:
\[
N_1 =
\left[
\begin{array}{rrrrrrrrrrrr}
\cdot & \cdot & \cdot & 1 & \cdot & 1 & \cdot & \cdot & \cdot & \cdot & \cdot & \cdot \\[-1pt]
-1 & \cdot & -1 & \cdot & \cdot & \cdot & \cdot & \cdot & \cdot & \cdot & \cdot & \cdot \\[-1pt]
\cdot & -1 & \cdot & \cdot & -1 & \cdot & \cdot & \cdot & \cdot & \cdot & \cdot & \cdot \\[-1pt]
\cdot & 1 & -1 & -1 & 1 & \cdot & -1 & 1 & 1 & \cdot & \cdot & \cdot \\[-1pt]
\cdot & \cdot & \cdot & -1 & 1 & \cdot & \cdot & 1 & \cdot & 1 & \cdot & \cdot \\[-1pt]
1 & \cdot & 1 & 1 & -1 & \cdot & 1 & -1 & \cdot & \cdot & 1 & \cdot \\[-1pt]
\cdot & \cdot & 1 & 1 & \cdot & \cdot & 1 & \cdot & \cdot & \cdot & \cdot & 1
\end{array}
\right]
\]
The rows of $N_1$ have squared Euclidean lengths $2, 2, 2, 4, 4, 7, 7$ (sorted) with product 6272.
We apply the LLL algorithm to the lattice $L$ spanned by the rows of
$N_1$, obtain a reduced basis of $L$, sort the vectors by increasing length,
and multiply each row by $\pm 1$ to make its leading entry positive:
\[
N_2 =
\left[
\begin{array}{rrrrrrrrrrrr}
\cdot & \cdot & \cdot & 1 & \cdot & 1 & \cdot & \cdot & \cdot & \cdot & \cdot & \cdot \\[-1pt]
\cdot & 1 & \cdot & \cdot & 1 & \cdot & \cdot & \cdot & \cdot & \cdot & \cdot & \cdot \\[-1pt]
1 & \cdot & 1 & \cdot & \cdot & \cdot & \cdot & \cdot & \cdot & \cdot & \cdot & \cdot \\[-1pt]
\cdot & \cdot & \cdot & \cdot & \cdot & \cdot & \cdot & 1 & 1 & \cdot & \cdot & 1 \\[-1pt]
\cdot & \cdot & 1 & \cdot & 1 & \cdot & \cdot & \cdot & -1 & \cdot & -1 & \cdot \\[-1pt]
\cdot & 1 & \cdot & \cdot & \cdot & -1 & \cdot & -1 & \cdot & -1 & \cdot & \cdot \\[-1pt]
\cdot & \cdot & 1 & \cdot & \cdot & -1 & 1 & -1 & -1 & \cdot & \cdot & \cdot
\end{array}
\right]
\]
The rows of $N_2$ have squared lengths $2, 2, 2, 3, 4, 4, 5$ with product 1920.
The lattice $L$ is a module over $\mathbb{Z} S_3$, where $S_3$ acts by permuting
$x,y,z$ in the ordered basis of $\mathcal{T}(1)$.
Rows $1, 4, 7$ of $N_2$ are the lex-minimal subset 
which generates $L$ as a $\mathbb{Z} S_3$-module.
These rows represent relations \eqref{alternating}--\eqref{comtrans}.
\end{proof}

\begin{remark}
There are 40 pairs of trilinear operations 
$xyz \pm x^p y^p z^p$ and $xyz \pm x^q y^q z^q$
in an associative triple system
where $p, q$ are distinct non-identity permutations of $x,y,z$.
In every case, a single operation generates the same $S_3$-module
as the original pair.
Up to equivalence there are four possibilities:
the translator $xyz - yzx$,
the weakly commutative operation $xyz - xzy + 2 zyx$,
the weakly anticommutative operation $xyz + xzy - 2 zyx$,
and the associative operation $xyz$.
\end{remark}


\section{Identities of degree 5}
\label{degree5section}

We first consider each operation separately and then the two operations together.

\begin{lemma}
\label{lemmadegree5commutator}
Every multilinear identity of degree 5 satisfied by the special commutator
in every associative triple system follows from relation \eqref{alternating} and
\begin{equation}
\label{degree5commutator}
T_{x,z}(v,w,y) + T_{x,z}(w,y,v) + T_{x,z}(y,v,w) = 0,
\end{equation}
where $T_{x,z}(v,w,y) = [[v,w,x],y,z] - [v,w,[x,y,z]]$.
(This simplifies the $q=2$ case of the deformed anti-Jordan triple product in \cite{BPCTO}.)
\end{lemma}

\begin{proof}
It is easy to check by hand that \eqref{degree5commutator} is satisfied by the special commutator.
We must verify that every identity in degree 5 satisfied by the special commutator follows from
\eqref{alternating} and \eqref{degree5commutator}.
We use computer algebra \cite{BPLLL}.

Let $\mathcal{T}$ be the free weight-graded operad generated by one ternary operation 
$[-,-,-]$ with no symmetry.
The homogeneous space $\mathcal{T}(2)$ is isomorphic to 
the direct sum of three copies of the regular representation $\mathbb{F} S_5$ 
corresponding to the three association types
$[[-,-,-],-,-]$, $[-,[-,-,-],-]$, $[-,-,[-,-,-]]$
in that order.
Each association type has 120 permutations of $v, w, x, y, z$ in lex order.

Each of these nonassociative monomials expands using the special commutator to 
a linear combination of monomials in the operad $\mathcal{A}$ of Definition \ref{defoperads}: 
for example,
\[
[[v,w,x],y,z] \; \mapsto \; vwxyz - wvxyz - yvwxz + ywvxz.
\]
We construct the $120 \times 360$ matrix $M$ whose $(i,j)$ entry is the coefficient of
the $i$-th associative monomial in the expansion of the $j$-th nonassociative monomial.

We compute the Hermite normal form $H$ of the transpose $M^t$ and an integer matrix $U$ 
with $\det(U) = \pm 1$ and $U M^t = H$.
We find that $H$ has rank 70 so the bottom 290 rows of $H$ are 0.
Hence the bottom 290 rows of $U$, denoted $N$, form an integer basis for the nullspace of $M$;
that is, the kernel of the expansion map $X_5$.
The largest squared Euclidean length of the rows of $N$ is 49352.

We apply the LLL algorithm to the rows of $N$ 
and obtain a matrix $N'$ whose rows generate the same lattice. 
The largest squared Euclidean length of the rows of $N'$ is only 6.
Let $N''$ consist of the rows of $N'$ sorted by increasing length.

Let $I(x,y,z) = [x,y,z] + [y,x,z]$ denote relation \eqref{alternating}.
We generate all consequences of $I$ in degree 5 by partial compositions in the operad;
that is, substituting $[-,-,-]$ into $I$, or substituting $I$ into $[-,-,-]$:
\begin{equation}
\label{consequences5}
\begin{array}{l@{\qquad}l@{\qquad}l}
I([x,v,w],y,z), &
I(x,[y,v,w],z), &
I(x,y,[z,v,w]), \\[2pt] {}
[I(x,y,z),v,w], &
[v,I(x,y,z),w], &
[v,w,I(x,y,z)].
\end{array}
\end{equation}
These relations generate the $S_5$-module of identities which follow from \eqref{alternating}.

We construct the $720 \times 360$ matrix $C$ whose rows contain all permutations of 
the relations \eqref{consequences5}.
We compute the RCF and find that $C$ has rank 270.
Since $N$ has rank 290, there is a 20-dimensional quotient $S_5$-module of new identities.

For each row $i$ of $N''$, we stack $C$ on top of the matrix containing all permutations
of the relation represented by row $i$.
Row 194 is the first which increases the rank from 270 to 290.
This row is the coefficient vector of relation \eqref{degree5commutator}.
\end{proof}

\begin{remark}
The special commutator is equivalent to $xyz + xzy - yxz + yzx - zxy - zyx$,
case $q = 2$ of the deformed anti-Jordan triple product \cite{BPCTO}.
This follows from the equality of the RCFs of the representation matrices for the operations:
\[
\left[ \;
0,
\left[ \begin{array}{cc} 1 & 2 \\ 0 & 0 \end{array} \right]\!\!, \,
1
\; \right]
\]
Relation \eqref{degree5commutator} is much simpler than 
the relation with 24 terms in \cite[equation (55)]{BPCTO}. 
\end{remark}

\begin{lemma}
\label{lemmadegree5translator}
Every multilinear identity in degree 5 satisfied by the special translator
in every associative triple system follows from relation \eqref{Jacobi} and
\begin{equation}
\label{degree5translator}
R_{y,z}( \langle v, w, x \rangle )
=
\langle R_{y,z}(v), w, x \rangle +
\langle v, R_{y,z}(w), x \rangle +
\langle v, w, R_{y,z}(x) \rangle,
\end{equation}
where $R_{y,z}(u) = \langle u, y, z \rangle$.
(Compare the result for the cyclic commutator in \cite{BPCTO}.)
\end{lemma}

\begin{proof}
Similar to the proof of Lemma \ref{lemmadegree5commutator}.
\end{proof}

\begin{theorem}
\label{degree5comtrans}
Every multilinear identity in degree 5 relating the special commutator and the special translator
in every associative triple system follows from 
\eqref{alternating}--\eqref{comtrans}, \eqref{degree5commutator}, \eqref{degree5translator},
and the following new identities involving both operations:
\begin{align}
&
  [ [ vwx ] yz ]
 + [ [ xvy ] wz ]
 - [ \langle vwx \rangle yz ]
 + [ \langle vyw \rangle xz ]
 - [ \langle xyw \rangle vz ]
= 0,
\label{degree5mixed1}
\\[-1pt]
&
  [ [vwx] yz ]
 - [ [vyw] xz ]
 + \langle [wvz] xy \rangle
 + [ \langle vyw \rangle xz ]
 + [ vw \langle zxy \rangle ]
= 0,
\label{degree5mixed2}
\\[-1pt]
&
  [[vwx]yz]
 + [[vwz]xy]
 - [[xwy]vz]
 - [[zwy]xv]
 + \langle [wvz]xy \rangle
\label{degree5mixed3}
\\[-1pt]
&
 + \langle [xwz]yv \rangle
 + \langle [zwy]xv \rangle
 + [\langle wvx \rangle yz]
 + [\langle wvz \rangle xy]
 - [\langle wyv \rangle xz]
\notag
\\[-1pt]
&
 - \langle \langle wvz \rangle xy \rangle
 + \langle \langle wyv \rangle xz \rangle
 + \langle wx\langle zyv \rangle \rangle
 = 0.
\notag
\end{align}
\end{theorem}

\begin{proof}
Let $\mathcal{T}$ be the free symmetric operad generated by 
two ternary operations $[-,-,-]$ and $\langle -,-,- \rangle$ with no symmetry.
The homogeneous space $\mathcal{T}(2)$ is isomorphic to the direct sum of 
12 copies of the regular representation $\mathbb{F} S_5$ corresponding to the association types
ordered as follows:
\[
\begin{array}{l@{\qquad}l@{\qquad}l@{\qquad}l}
[[-,-,-],-,-], &
\langle [-,-,-],-,- \rangle, &
[ \langle -,-,- \rangle ,-,-], &
\langle \langle -,-,- \rangle,-,- \rangle,
\\ {}
[-,[-,-,-],-], &
\langle -,[-,-,-],- \rangle, &
[-,\langle -,-,- \rangle,-], &
\langle -,\langle -,-,- \rangle,- \rangle,
\\ {}
[-,-,[-,-,-]], &
\langle -,-,[-,-,- \rangle, &
[-,-,\langle -,-,- \rangle ], &
\langle -,-,\langle -,-,- \rangle \rangle.
\end{array}
\]
Each association type has 120 permutations of $v, w, x, y, z$ in lex order.
We construct the $120 \times 1440$ matrix $M$ in which the $(i,j)$ entry is the coefficient of
the $i$-th associative monomial in the expansion of the $j$-th nonassociative monomial.
There are 36 consequences of identities \eqref{alternating}--\eqref{comtrans} since we can substitute
either of two operations.
We must exclude these consequences together with all permutations of relations
\eqref{degree5commutator} and \eqref{degree5translator}.
The rest is similar to the proof of Lemma \ref{lemmadegree5commutator}.
\end{proof}

\begin{definition}
A \emph{Smith algebra} is a comtrans algebra satisfying the five identities of
Lemmas \ref{lemmadegree5commutator}, \ref{lemmadegree5translator} and
Theorem \ref{degree5comtrans}.
(Since these identities have weights 1 and 2, the corresponding operad is quadratic
and hence has a Koszul dual \cite{GK}.)
\end{definition}

\begin{remark}
The weakly anticommutative operation 
satisfies a 141-dimensional $S_5$-module of multilinear identities
which are not consequences of the symmetric sum identity (Remark \ref{wacremark}).
If $T_{yz}(u) = \{u,y,z\} + \{u,z,y\}$ then
the simplest new identity that we found (which however does not generate all new identities) is
\[
T_{yz}( \{v,w,x\} ) =
\{ v, T_{yz}(w), x \} + \{ v, T_{yz}(w), x \} + \{ v, w, T_{yz}(x) \}.
\]
\end{remark}


\section{Identities of degree 7}
\label{degree7section}

Every special comtrans algebra is a Smith algebra, but the converse is false:
we demonstrate the existence of identities in degree 7 satisfied by every special
comtrans algebra but not by every Smith algebra.
There are new identities for each operation separately, but none relating the operations.

For degree 7, the methods of the previous sections are impractical: 
for one (resp. two) operation(s), there are 12 (resp.~96) association types
and 60480 (resp. 483840) multilinear monomials.
We use a constructive version of the representation theory of
the symmetric group to decompose the kernel of the expansion map $X_7$
into isotypic components corresponding to the partitions of 7.
We provide only an outline; 
these methods have been discussed in detail in previous papers \cite{BBM, BMP, BPHITC}.
We order the 12 association types for one ternary operation as follows:
\begin{equation}
\label{degree7types}
\begin{array}{l@{\quad}l@{\quad}l}
[[[-, -, -], -, -], -, -], &
[[-, [-, -, -], -], -, -], &
[[-, -, [-, -, -]], -, -],
\\[1pt] {}
[[-, -, -], [-, -, -], -], &
[[-, -, -], -, [-, -, -]], &
[-, [[-, -, -], -, -], -],
\\[1pt] {}
[-, [-, [-, -, -], -], -], &
[-, [-, -, [-, -, -]], -], &
[-, [-, -, -], [-, -, -]],
\\[1pt] {}
[-, -, [[-, -, -], -, -]], &
[-, -, [-, [-, -, -], -]], &
[-, -, [-, -, [-, -, -]]].
\end{array}
\end{equation}
Recall that an identity is \emph{new} if it does not follow from those of lower degree.

\begin{lemma}
\label{commutatordegree7}
The $S_7$-module of new identities for the special commutator is nonzero,
and has the following decomposition into isotypic components:
\[
[43] \oplus [421]^2 \oplus [41^3] \oplus [3^21]^3 \oplus [32^2]^2 \oplus
[321^2]^5 \oplus [31^4]^2 \oplus [2^31]^2 \oplus [2^21^3]^3 \oplus [21^5].
\]
We write $[\lambda]^m$ if the irreducible representation for partition $\lambda$
has multiplicity $m$.
\end{lemma}

\begin{proof}
There are six consequences \eqref{consequences5} in degree 5 of identity \eqref{alternating}.
Combining these with \eqref{degree5commutator}, we obtain seven identities $J(v,w,x,y,z)$
each of which produces eight consequences $K(t,\dots,z)$ in degree 7;
these can be expressed using partial compositions as $J \circ_k \gamma$ ($1 \le k \le 5$) 
and $\gamma \circ_k J$ ($1 \le k \le 3$) where $\gamma$ denotes the commutator.
For each partition $\lambda$ of 7 we write
$R_\lambda \colon \mathbb{Q} S_7 \to M_{d_\lambda}(\mathbb{Q})$
for the corresponding irreducible representation;
$d_\lambda$ is the number of standard tableaux of shape $\lambda$.
For each $p \in S_7$ we use Clifton's algorithm \cite{BMP,C} to compute $R_\lambda(p)$.
For each $\lambda$ we construct a $56 d_\lambda \times 12 d_\lambda$ matrix $C_\lambda$
partitioned into $d_\lambda \times d_\lambda$ blocks.
Each consequence $K_i$ ($1 \le i \le 56$) is a sum of 12 components $K_i^1, \dots K_i^{12}$
corresponding to the association types \eqref{degree7types}.
Block $(i,j)$ of $C_\lambda$ contains $R_\lambda( K_i^j )$.
We compute the RCF of $C_\lambda$; its rank $c_\lambda$ is 
the multiplicity of representation $[\lambda]$ in the $S_7$-module of all consequences 
of the identities of lower degree (Figure \ref{c7m}).

We substitute the identity permutation of $t, u, \dots, z$ into each association type,
and obtain monomials $\xi_1, \dots, \xi_{12}$ which we expand using the special commutator
into a linear combination of eight associative monomials.
For each $\lambda$ we construct a $d_\lambda \times 12 d_\lambda$ matrix $E_\lambda$: 
one row of $d_\lambda \times d_\lambda$ blocks in which block $j$ is $R_\lambda( \xi_j )^t$.
We compute the RCF of $E_\lambda$ and find its rank $e_\lambda$;
then $a_\lambda = 12 d_\lambda - e_\lambda$
is the multiplicity of $[\lambda]$ in the kernel of the expansion map (Figure \ref{c7m}).
We compute the $a_\lambda \times 12 d_\lambda$ matrix $N_\lambda$ in RCF whose row space is
the nullspace of $E_\lambda$, and check that the row space of $N_\lambda$ contains
the row space of $C_\lambda$.
The difference $n_\lambda = a_\lambda - c_\lambda$ is the multiplicity of $[\lambda]$
in the quotient module $N_\lambda / C_\lambda$ (we identify each matrix with its row space)
of new identities in degree 7 for partition $\lambda$ (Figure \ref{c7m}).
\end{proof}

\begin{figure}[ht]
$
\begin{array}
{r|r@{\;\;\,}r@{\;\;\,}r@{\;\;\,}r@{\;\;\,}r@{\;\;\,}r@{\;\;\,}r@{\;\;\,}r@{\;\;\,}r@{\;\;\,}r@{\;\;\,}r@{\;\;\,}r@{\;\;\,}r@{\;\;\,}r@{\;\;\,}r}
\lambda &
7 & 61 & 52 & 51^2 & 43 & 421 & 41^3 & 3^21 & 32^2 & 321^2 & 31^4 & 2^31 & 2^21^3 & 21^5 & 1^7
\\
\midrule
d_\lambda &
1 & 6 & 14 & 15 & 14 & 35 & 20 & 21 & 21 & 35 & 15 & 14 & 14 & 6 & 1
\\
c_\lambda &
12 & 71 & 162 & 173 & 157 & 394 & 225 & 231 & 233 & 385 & 166 & 153 & 152 & 66 & 11
\\
a_\lambda &
12 & 71 & 162 & 173 & 158 & 396 & 226 & 234 & 235 & 390 & 168 & 155 & 155 & 67 & 11
\\
n_\lambda &
0 & 0 & 0 & 0 & 1 & 2 & 1 & 3 & 2 & 5 & 2 & 2 & 3 & 1 & 0
\\
\midrule
\end{array}
$
\vspace{-4mm}
\caption{Multiplicities in degree 7 for the ternary commutator}
\label{c7m}
\end{figure}

\begin{lemma}
\label{translatordegree7}
The $S_7$-module of new identities for the special translator is nonzero
and has the following decomposition into isotypic components:
\[
[52] \oplus [51^2] \oplus [43] \oplus [421]^3 \oplus [41^3]^2 \oplus [3^21]^2 \oplus
[32^2]^2 \oplus [321^2]^3 \oplus [31^4] \oplus [2^31] \oplus [2^21^3].
\]
\end{lemma}

\begin{proof}
Similar to the proof of Lemma \ref{commutatordegree7}; see Figure \ref{t7m}.
\end{proof}

\begin{figure}[ht]
$
\begin{array}
{r|r@{\;\;\,}r@{\;\;\,}r@{\;\;\,}r@{\;\;\,}r@{\;\;\,}r@{\;\;\,}r@{\;\;\,}r@{\;\;\,}r@{\;\;\,}r@{\;\;\,}r@{\;\;\,}r@{\;\;\,}r@{\;\;\,}r@{\;\;\,}r}
\lambda &
7 & 61 & 52 & 51^2 & 43 & 421 & 41^3 & 3^21 & 32^2 & 321^2 & 31^4 & 2^31 & 2^21^3 & 21^5 & 1^7
\\
\midrule
d_\lambda &
1 & 6 & 14 & 15 & 14 & 35 & 20 & 21 & 21 & 35 & 15 & 14 & 14 & 6 & 1
\\
c_\lambda &
12 & 68 & 156 & 168 & 155 & 388 & 222 & 232 & 232 & 388 & 168 & 155 & 156 & 68 & 12
\\
a_\lambda &
12 & 68 & 157 & 169 & 156 & 391 & 224 & 234 & 234 & 391 & 169 & 156 & 157 & 68 & 12
\\
n_\lambda &
0 & 0 & 1 & 1 & 1 & 3 & 2 & 2 & 2 & 3 & 1 & 1 & 1 & 0 & 0
\\
\midrule
\end{array}
$
\vspace{-4mm}
\caption{Multiplicities in degree 7 for the ternary translator}
\label{t7m}
\end{figure}

For two ternary operations we have $2^3 \cdot 12 = 96$ association types:
in each type for one operation \eqref{degree7types}, 
each operation may be replaced by either of two operations.

\begin{theorem}
Every identity in degree 7 satisfied by the special commutator and translator 
in every associative triple system follows from
\eqref{alternating}--\eqref{comtrans},
\eqref{degree5commutator}, \eqref{degree5translator},
and \eqref{degree5mixed1}--\eqref{degree5mixed3},
together with the new identities for the commutator and translator separately
whose existence is demonstrated by Lemmas \ref{commutatordegree7} and \ref{translatordegree7}.
That is, special comtrans algebras satisfy no new identities
in which every term has both operations.
\end{theorem}

\begin{proof}
Each identity $J$ in equations \eqref{alternating}--\eqref{comtrans}
produces 12 consequences in degree 5:
the partial compositions
$J \circ_k \gamma$, $J \circ_k \delta$ ($1 \le k \le 3$),
$\gamma \circ_k J$, $\gamma \circ_k J$ ($1 \le k \le 3$),
where $\gamma$, $\delta$ denote commutator and translator.
Including \eqref{degree5commutator}, \eqref{degree5translator},
\eqref{degree5mixed1}--\eqref{degree5mixed3}, we obtain 41 consequences $K$ in degree 5.
Each of these produces 16 consequences in degree 7:
$K \circ_k \gamma$, $K \circ_k \delta$ ($1 \le k \le 5$),
$\gamma \circ_\ell K$, $\gamma \circ_\ell K$ ($1 \le \ell \le 3$).
For each partition $\lambda$ of 7, the matrix $C_\lambda$ representing the consequences
has size $656 d_\lambda \times 96 d_\lambda$.
Let $D^1_\lambda$, $D^2_\lambda$ be obtained from the matrices $C_\lambda$ in the proofs of 
Lemmas \ref{commutatordegree7}, \ref{translatordegree7} by embedding the 12 association types
for one operation into the 96 association types for two operations.
The row space of $D^1_\lambda$ (resp. $D^2_\lambda$) contains the identities
satisfied by the commutator (resp. translator) in degree 7.
We stack $C_\lambda$ on top of $D^1_\lambda$ and $D^2_\lambda$ to obtain $CD_\lambda$;
we denote its rank by $c_\lambda$.
The expansion matrix $E_\lambda$ has size $d_\lambda \times 96 d_\lambda$;
we denote its rank by $e_\lambda$, so it has nullity $a_\lambda = 96 d_\lambda - e_\lambda$.
Let $N_\lambda$ be the matrix whose row space is the nullspace of $E_\lambda$.
For each partition $\lambda$, we find that $c_\lambda = a_\lambda$ and that 
the RCFs of $C_\lambda$ and $N_\lambda$ coincide.
\end{proof}


\section{Enveloping algebras for $2 \times 2$ matrices}
\label{envelopesection}

In this final section we use noncommutative Gr\"obner bases \cite{BF} 
to construct the universal associative enveloping algebras of 
the nonassociative triple systems $A^C$, $A^T$, and $A^{CT}$,
obtained by applying the special commutator, special translator, and both together,
to the associative triple system $A = ( a_{ij} )$ of $2\times 2$ matrices.
This method has previously been used  
for the universal envelopes of triple systems \cite{E2}
obtained by applying trilinear operations \cite{BPCTO}
to the $2 \times 2$ matrices with $a_{11} = a_{22} = 0$,
and for infinite families of simple anti-Jordan triple systems \cite{E1,E3}.
This approach to the representation theory of comtrans algebras
is based on the special commutator and special translator in an associative triple system,
and therefore differs essentially from the approach of 
Im, Shen \& Smith \cite{IM2009,IM2011,SS1992}.

\begin{definition}
Let $A$ be the associative triple system of $2\times 2$ matrices.
Let $B$ and $X$ be the basis of matrix units
and a set of symbols in bijection with $B$:
\[
B = \{ E_{ij} \mid 1 \le i,j \le 2 \},
\qquad
X = \{ e_{ij} \mid 1 \le i,j \le 2 \},
\qquad
\eta( E_{ij} ) = e_{ij}.
\]
Extend $\eta$ linearly to the injective map $\eta\colon A \to F\langle X\rangle$
where $F\langle X\rangle$ is the free associative algebra generated by $X$.
Define three ideals in $F\langle X\rangle$ as follows:
\begin{alignat*}{2}
&
I^C = \langle G^C \rangle,
&\qquad
&
G^C = \{
e_{ij} e_{kl} e_{st} - e_{kl} e_{ij} e_{st} - \eta([ E_{ij}, E_{kl}, E_{st}])
\},
\\
&
I^T = \langle G^T \rangle,
&\qquad
&
G^T = \{
e_{ij} e_{kl} e_{st} - e_{kl} e_{st} e_{ij} - \eta (\langle E_{ij}, E_{kl}, E_{st} \rangle)
\},
\\
&
I^{CT} = \langle G^{CT} \rangle,
&\qquad
&
G^{CT} = G^T \cup G^C.
\end{alignat*}
The corresponding universal associative enveloping algebras are
\[
U(A^C) = F\langle X\rangle/ I^C,
\qquad
U(A^T) = F\langle X\rangle/ I^T,
\qquad
U(A^{CT}) = F\langle X\rangle/ I^{CT}.
\]
We write $e_{ij} \prec e_{kl}$ if $i < k$, or $i = k$, $j < l$;
and also $e_{11}, e_{12}, e_{21}, e_{22} = a, b, c, d$.
\end{definition}

\begin{lemma}\label{AC}
The universal enveloping algebra $U(A^C)$ has basis
\[
\mathfrak{B}^C = \{ \, 1, \, a,\, b,\, c,\, d,\, a^2,\, ab, \, ca,\, cb \, \}.
\]
\end{lemma}

\begin{proof}
The set $G^T$ contains 64 elements; after putting each generator in standard form
(reverse deglex order and monic), only 24 distinct generators remain:
\[
\begin{array}{l@{\quad}l@{\quad}l@{\quad}l}
G_1 = ba^2 {-} aba, & G_2 = bab {-} ab^2, & G_3 = bac {-} abc {+} a, & G_4 = bad {-} abd {+} b,
\\
G_5 = ca^2 {-} aca {-} c, & G_6 = cab {-} acb {-} d, & G_7 = cac {-} ac^2, & G_8 = cad {-} acd,
\\
G_9 = cba {-} bca {+} a, & G_{10} = cb^2 {-} bcb {+} b, & G_{11} = cbc {-} bc^2 {-} c, & G_{12} = cbd {-} bcd {-} d,
\\
G_{13} = da^2 {-} ada, & G_{14} = dab {-} adb, & G_{15} = dac {-} adc, & G_{16} = dad {-} ad^2,
\\
G_{17} = dba {-} bda, & G_{18} = db^2 {-} bdb, & G_{19} = dbc {-} bdc {+} a, & G_{20} = dbd {-} bd^2 {+} b,
\\
G_{21} = dca {-} cda {-} c, & G_{22} = dcb {-} cdb {-} d, & G_{23} = dc^2 {-} cdc, & G_{24} = dcd {-} cd^2.
\end{array}
\]
There are only 56 distinct normal forms of the compositions of these generators;
the corresponding compositions appear in Figure \ref{GBIC}.

\begin{figure}[ht]
$
\begin{array}{l@{\,\,}l@{\,\,}l@{\,\,}l}
\mathcal{S}_1 {\,=\,} G_3 a^2 {-} ba G_5{,} &
\mathcal{S}_2 {\,=\,} G_3 ab {-} ba G_6{,} &
\mathcal{S}_3 {\,=\,} G_3 ac {-} ba G_7{,} &
\mathcal{S}_4 {\,=\,} G_3 ad {-} ba G_8{,}
\\
\mathcal{S}_5 {\,=\,} G_3 ba {-} ba G_9{,} &
\mathcal{S}_6 {\,=\,} G_3 b^2 {-} ba G_{10}{,} &
\mathcal{S}_7 {\,=\,} G_3 bc {-} ba G_{11}{,} &
\mathcal{S}_8 {\,=\,} G_3 bd {-} ba G_{12}{,}
\\
\mathcal{S}_9 {\,=\,} G_4 ba {-} ba G_{17}{,} &
\mathcal{S}_{10} {\,=\,} G_{4} b^2 {-} ba G_{18}{,} &
\mathcal{S}_{11} {\,=\,} G_{4}bc {-} ba G_{19}{,} &
\mathcal{S}_{12} {\,=\,} G_4 bd {-} ba G_{20}{,}
\\
\mathcal{S}_{13} {\,=\,} G_4 ca {-}ba G_{21}{,} &
\mathcal{S}_{14} {\,=\,} G_4 cb {-} ba G_{22}{,} &
\mathcal{S}_{15} {\,=\,} G_{4} c^2 {-} ba G_{23}{,} &
\mathcal{S}_{16} {\,=\,} G_{4}cd {-}bs G_{24}{,}
\\
\mathcal{S}_{17} {\,=\,} G_7 a^2 {-} ca G_5{,} &
\mathcal{S}_{18} {\,=\,} G_7 ab {-} ca G_6{,} &
\mathcal{S}_{19} {\,=\,} G_7 ac {-} ca G_7{,} &
\mathcal{S}_{20} {\,=\,} G_7 ad {-} ca G_8{,}
\\
\mathcal{S}_{21} {\,=\,} G_7 ba {-} ca G_9{,} &
\mathcal{S}_{22} {\,=\,} G_7 b^2 {-} ca G_{10}{,} &
\mathcal{S}_{23} {\,=\,} G_{7} bc {-} ca G_{11}{,} &
\mathcal{S}_{24} {\,=\,} G_7 bd {-} ca G_{12}{,}
\\
\mathcal{S}_{25} {\,=\,} G_8 ba {-} ca G_{17}{,} &
\mathcal{S}_{26} {\,=\,} G_{8} b^2 {-} ca G_{18}{,} &
\mathcal{S}_{27} {\,=\,} G_{8} bc {-} ca G_{19}{,} &
\mathcal{S}_{28} {\,=\,} G_8 bd {-}ca G_{20}{,}
\\
\mathcal{S}_{29} {\,=\,} G_9 a {-} c G_1{,} &
\mathcal{S}_{30} {\,=\,} G_9 b {-} c G_2{,} &
\mathcal{S}_{31} {\,=\,} G_9 c {-} c G_3{,} &
\mathcal{S}_{32} {\,=\,} G_9 d {-} c G_4{,}
\\
\mathcal{S}_{33} {\,=\,} G_{11} ba {-} cb G_9{,} &
\mathcal{S}_{34} {\,=\,} G_{11} b^2 {-} cb G_{10}{,} &
\mathcal{S}_{35} {\,=\,} G_{11} bc {-} cb G_{11}{,} &
\mathcal{S}_{36} {\,=\,} G_{11} bd {-}cb G_{12}{,}
\\
\mathcal{S}_{37} {\,=\,} G_{12} a^2 {-} cb G_{13}{,} &
\mathcal{S}_{38} {\,=\,} G_{12} ab {-} cb G_{14}{,} &
\mathcal{S}_{39} {\,=\,} G_{12} ac {-} cb G_{15}{,} &
\mathcal{S}_{40} {\,=\,} G_{12} ad {-} cb G_{16}{,}
\\
\mathcal{S}_{41} {\,=\,} G_{12} ba {-} cb G_{17}{,} &
\mathcal{S}_{42} {\,=\,} G_{12} b^2 {-} cb G_{18}{,} &
\mathcal{S}_{43} {\,=\,} G_{12} bc {-} cb G_{19}{,} &
\mathcal{S}_{44} {\,=\,} G_{12} bd {-} cb G_{20}{,}
\\
\mathcal{S}_{45} {\,=\,} G_{17}a {-} d G_1{,} &
\mathcal{S}_{46} {\,=\,} G_{17} b {-} d G_{2}{,} &
\mathcal{S}_{47} {\,=\,} G_{21} a {-} d G_5{,} &
\mathcal{S}_{48} {\,=\,} G_{21} b {-} d G_6{,}
\\
\mathcal{S}_{59} {\,=\,} G_{21} c {-} d G_{7}{,} &
\mathcal{S}_{50} {\,=\,} G_{21} d {-} d G_8{,} &
\mathcal{S}_{51} {\,=\,} G_{22} a^2 {-} dc G_1{,} &
\mathcal{S}_{52} {\,=\,} G_{22} ab {-} dc G_2{,}
\\
\mathcal{S}_{53} {\,=\,} G_{22} ac {-} dc G_3{,} &
\mathcal{S}_{54} {\,=\,} G_{22} ad {-} dc G_4{,} &
\mathcal{S}_{55} {\,=\,} G_{22} a {-} d G_9{,} &
\mathcal{S}_{56} {\,=\,} G_{22} b {-} d G_{10}.
\end{array}
$
\vspace{-4mm}
\caption{Compositions of the generators for the ideal $I^C \subset F\langle X \rangle$}
\label{GBIC}
\end{figure}

To compute normal forms we eliminate all occurrences of leading monomials of the generators.
We write $\equiv$ for congruence modulo $G_1$, \dots, $G_{24}$.
For example, to find the normal form of $\mathcal{S}_1$,
we eliminate the leading monomials of $G_5$, $G_1$, $G_3$:
\begin{align*}
\mathcal{S}_1
&=
(bac - abc + a)a^2- ba (ca^2 - aca - c)
=
(- abc + a)a^2- ba (- aca - c)
\\[-1mm]
&\equiv
- ab(aca+c) + a^3 + (aba)ca + ba c
\\[-1mm]
&\equiv
- a ((abc-a) a+bc) + a^3 + a(abc-a) a + (abc-a)
=  a^3 -a.
\end{align*}
Similar calculations produce the normal forms of
$\mathcal{S}_2$, \dots $\mathcal{S}_{56}$:
\[
\begin{array}{l}
a^2b-b,\;
a^2c, \;
a^2d, \;
aba, \;
ab^2, \;
abc {-} a, \;
abd {-} b, \;
b^2a, \;
b^3, \;
b^2c, \;
b^2d, \;
bca {-} ada {-} a,
\\
bcb {-} adb {-} b, \;
bc^2 {-} adc, \;
bcd {-} ad^2, \;
c^2a, \;
c^2b, \;
c^3, \;
c^2d, \;
cda, \;
cdb, \;
cdc, \;
cd^2, \;
d^2a, \;
d^2b,
\\
d^2c {-} c, \;
d^3 {-} d, \;
da {-} bc {+} aa, \;
db {-} bd {+} ab, \;
dc {-} ca {+}3 ac, \;
dd {-} cb {+}3 ad, \;
bca {-} a, \;
bcb {-} b,
\\
bc^2, \;
bcd, \;
d^2a {-} bca {+} ada {+} a, \;
d^2b {-} bcb {+} adb {+} b, \;
d^2c {-} bc^2 {+} adc {-} c, \;
d^3 {-} bcd {+} ad^2 {-} d,
\\
bda, \;
bdb, \;
bdc {-} a, \;
bd^2 {-} b, \;
ba, \; b^2, \;
dc {-} ca {+} ac, \;
d^2{-} cb {+} ad, \;
c^2, \;
cd, \;
\\
d^2a {-}3 ada {+} abc {-} a^3, \;
d^2b {-}3 adb {+} abd {-} a^2b, \;
d^2c {-} adc {+} aca {-} a^2c {-} c,
\\
d^3 {-} ad^2 {+} acb {-} a^2d {-} d, \;
da {-} \tfrac{1}{3} bc {+}\tfrac{1}{3} a^2, \;
db {-}\tfrac{1}{3} bd {+}\tfrac{1}{3} ab.
\end{array}
\]
Including these normal forms with the original 24 generators, and then self-reducing, 
produces the following 16 generators:
\begin{equation}
\label{UACGB}
\begin{array}{llllllll}
ac, & ad, & ba, & b^2, & bc{-}a^2, & bd{-}ab, & c^2, & cd,
\\
da, & db, & dc{-}ca, & d^2{-}cb, & a^3{-}a, & a^2b{-}b, & ca^2{-}c, & cab{-}d.
\end{array}
\end{equation}
All compositions of these 16 generators reduce to 0, 
and so we have a Gr\"obner basis for $I^C$.
A basis for $U(A^C)$ consists of the cosets of those monomials
which are not divisible by the leading monomial of any element of the Gr\"obner basis.
\end{proof}

\begin{lemma}\label{rr}
In $U(A^C )$ we have the relations
\[
\begin{array}{l}
ac = ad = ba = b^2 = cd = da = db = c^2 = 0,
\\
bc = a^2, \;\; bd = ab, \;\; dc = ca, \;\; d^2 = cb, \;\; 
a^3 = a, \;\; a^2b = b, \;\; ca^2= c, \;\; cab = d.
\end{array}
\]
\end{lemma}

\begin{proof}
This follows immediately from the Gr\"obner basis \eqref{UACGB}.
\end{proof}

\begin{lemma}
The nonzero structure constants of $U(A^C)$ are
\[
\begin{array}{l@{\;\;\;}l@{\;\;\;}l@{\;\;\;}l@{\;\;\;}l@{\;\;\;}l@{\;\;\;}l}
a \cdot a = a^2, &
a\cdot b = ab, &
a\cdot a^2 = a, &
a\cdot ab = b, &
b\cdot c = a^2, &
b\cdot d = ac,
\\
b\cdot ca = a, &
b\cdot cb = b, &
c\cdot a = ca, &
c\cdot b = cb, &
c\cdot a^2 = c, &
c\cdot ab =d,
\\
d\cdot c = ca, &
d\cdot d = cb, &
d\cdot ca = c, &
d\cdot cb = d, &
a^2\cdot a = a, &
a^2 \cdot b = b,
\\
a^2 \cdot a^2 = a^2, &
a^2 \cdot ab = ab, &
ab \cdot c = a, &
ab \cdot d = b, &
ab \cdot ca = a^2, &
ab\cdot cb = ab,
\\
ca \cdot a = c, &
ca \cdot b = d, &
ca \cdot a^2 = ca, &
ca \cdot ab = cb, &
cb\cdot c = c, &
ca \cdot d = d,
\\
cb \cdot ca = ca, &
cb \cdot cb = cb.
\end{array}
\]
\end{lemma}

\begin{proof}
This follows immediately from Lemma \ref{rr}.
\end{proof}

\begin{theorem}
The Wedderburn decomposition of $U(A^C)$ is
\[
U(A^C) = \rational \oplus M_{2\times 2}(\rational) \oplus M_{2\times 2}(\rational),
\]
where $M_{2\times 2}(\rational)$ is the ordinary associative algebra of $2\times 2$ matrices.
\end{theorem}

\begin{proof}
Following \cite{BW},
we first verify that the radical is zero and hence $U(A^C)$ is semisimple.
The center $Z(U(A^C))$ has this basis and structure constants:
\[
\begin{array}{c}
z_1 = 1, \qquad z_2 = a+d, \qquad z_3 = a^2+cb,
\\
z_1 \cdot z_1 = z_1, \;\;
z_1 \cdot z_2 = z_1, \;\;
z_1 \cdot z_3 = z_3, \;\;
z_2 \cdot z_2 = z_3, \;\;
z_2 \cdot z_3 = z_2, \;\;
z_3 \cdot z_3 = z_3.
\end{array}
\]
The minimal polynomial of $z_3$ is $t^2- t$ and hence
$Z(U(A^C))= J \oplus K$ where $J = \langle z_3-z_1 \rangle$ and $K = \langle z_3 \rangle$.
We have $\dim J = 1$ with basis $z_3-z_1$, and $\dim K = 2$ with basis $z_2$, $z_3$.
In $K$ the identity element is $z_3$, and the minimal polynomial of $z_2$ is $t^2 - z_3$.
Hence $K$ splits into 1-dimensional ideals with bases $z_2- z_3$ and $z_2+z_3$.
After scaling, we obtain a basis of orthogonal idempotents for $Z(U(A^C))$:
\[
e_1 = z_1- z_3, \qquad e_2 = \tfrac12 ( z_2-z_3 ), \qquad e_3 = \tfrac12 ( z_2 + z_3 ).
\]
These elements of the center correspond to the following elements of $U(A^C)$:
\[
e_1 = 1-a^2+cb, \qquad e_2 = \tfrac12 (a+d-a^2+cb), \qquad e_3= \tfrac12 (a+d+a^2+cb).
\]
The ideals in $U(A^C)$ generated by $e_1$, $e_2$, $e_3$ have dimensions 1, 4, 4 respectively,
which completes the proof.
(We omit the construction of an isomorphism between each simple two-sided ideal in $U(A^C)$
and the corresponding matrix algebra.)
\end{proof}

\begin{lemma}
\label{Ba}
The universal enveloping algebra $U(A^T)$ has basis
\[
\mathfrak{B}^T =
\{ \, a^m, \, b, \, c, \, ab, \, ca, \, cb, \, a^n d \mid m \ge 0, \, 0 \le n \le 2 \, \}.
\]
\end{lemma}

\begin{proof}
The set $G^T$ contains 64 elements; we put each generator in standard form
(reverse deglex order and monic) and self-reduce the set, obtaining 40 generators:
\[
\begin{array}{l@{\quad}l@{\quad}l@{\quad}l@{\quad}l}
aba{-}a^2b{+}b, & aca{-}a^2c, & ada{-}a^2d, & ba^2{-}a^2b{+}b, & bab{-}ab^2,
\\
bac{-}acb, & bad{-}adb, & b^2a{-}ab^2, & bca{-}abc, & bcb{-}b^2c{-}b,
\\
bda{-}abd{+}b, & bdb{-}bbd, & ca^2{-}a^2c{-}c, & cab{-}abc{-}d{+}a, & cac{-}ac^2,
\\
cad{-}adc, & cba{-}acb, & cb^2{-}b^2c, & cbc{-}bc^2{-}c, & cbd{-}bdc{-}d{+}a,
\\
c^2a{-}ac^2, & c^2b{-}bcc, & cda{-}acd, & cdb{-}bcd, & cdc{-}c^2d,
\\
da^2{-}a^2d, & dab{-}abd{+}b, & dac{-}acd, & dad{-}ad^2, & dba{-}adb,
\\
db^2{-}b^2d, & dbc{-}bcd, & dbd{-}bd^2{+}b, & dca{-}adc{-}c, & dcb{-}bdc{-}d{+}a,
\\
dc^2{-}c^2d, & dcd{-}cd^2, & d^2a{-}ad^2, & d^2b{-}bd^2{+}b, & d^2c{-}cd^2{-}c.
\end{array}
\]
There are 143 nontrivial compositions of these generators;
including their normal forms and self-reducing produces 16 generators:
\begin{equation}
\label{reT}
\begin{array}{l}
ac, \quad ba, \quad b^2, \quad bc{+}ad{-}a^2, \quad bd{-}ab, \quad c^2, \quad cd,
\\
da{-}ad, \quad db, \quad dc{-}ca, \quad d^2{-}cb{-}ad, \quad a^2b{-}b, \quad ca^2{-}c,
\\
cab{+}a^2d{-}a^3{-}d{+}a, \quad cad, \quad a^3d{-}a^4{-}ad{+}a^2.
\end{array}
\end{equation}
All compositions of these generators reduce to 0, so we have a Gr\"obner basis of $I^T$.
A basis for $U(A^T)$ consists of the cosets of those monomials which are not divisible by
the leading monomial of any element of the Gr\"obner basis.
\end{proof}

\begin{lemma}\label{cv}
In $U(A^T)$, we have the relations
\[
\begin{array}{l}
ac = ba = b^2 = c^2 = cd = db = cad = 0,
\\
bc = -ad+a^2, \quad bd = ab, \quad da = ad, \quad dc = ca, \quad d^2 = cb+ad,
\\
a^2b = b, \quad ca^2 = c, \quad cab = -a^2d+a^3+d-a, \quad a^3d = a^4+ad-a^2.
\end{array}
\]
\end{lemma}

\begin{proof}
This follows immediately from the Gr\"obner basis \eqref{reT}.
\end{proof}

\begin{theorem}
For  $m, l \ge 0$ and $0 \le n \le 2$, the structure constants of $U(A^T)$ are
\begin{align}
\label{a}
&
a^m \cdot a^l = a^{m+n},
\qquad
a^m \cdot b =
  \begin{cases}
  ab & ( m \, \text{odd} ) \\[-1mm]
  b  & ( m \, \text{even} ),
  \end{cases}
\\
\label{b}
&
a^m \cdot a^n d =
  \begin{cases}
  a^{m+n} d & ( m{+}n < 3 ) \\[-1mm]
  a^{m+n+1}+ a^{m+n-2} d - a^{m+n-1} & ( m{+}n \ge 3 ),
  \end{cases}
\\
\notag
&
b \cdot c =  -ad + a^2,
\qquad
b \cdot ca = -a^2d+ a^3,
\qquad
b \cdot cb = b,
\qquad
b \cdot d=  ab,
\\
\notag
&
c \cdot a^m =
  \begin{cases}
  ca & ( m \, \text{odd} ) \\[-1mm]
  c  & ( m \, \text{even} ),
  \end{cases}
\qquad
c \cdot b = cb,
\qquad
c \cdot (ab) =  -a^2 d+ a^3+ d-a,
\\
\notag
&
ab \cdot c = - a^2d + a^3,
\qquad
ab \cdot ca = -ad + a^2,
\qquad
ab \cdot  cb = ab,
\\
\notag
&
ca \cdot b = - a^2 d+a^3+d-a,
\qquad
ca \cdot ab = cb,
\\
\notag
&
cb \cdot c = c,
\qquad
cb \cdot ca= ca,
\qquad
cb \cdot cb=cb,
\qquad
cb \cdot d =  -a^2 d+ a^3+ d-a,
\\
\notag
&
a^n d \cdot a^m =
  \begin{cases}
  a^{m+n} d & ( m{+}n < 3 ) \\[-1mm]
  a^{m+n+1}+ a^{m+n-2} d - a^{m+n-1} & ( m{+}n \ge 3 ),
  \end{cases}
\\
\notag
&
d \cdot c = ca,
\qquad
d^2 = cb+ad,
\qquad
d \cdot ca = c,
\qquad
d \cdot cb = -a^2 d+a^3+d-a,
\\
\notag
&
ad \cdot ca  = ac,
\qquad
ad \cdot d= ad^2,
\qquad
ad \cdot ad= a^4 +ad-a^2,
\\
\label{c}
&
ad \cdot a^2d =  a^5- a^3+a^2d,
\qquad
a^2d \cdot a^2d= a^6 +ad -a^2.
\end{align}
\end{theorem}

\begin{proof}
The first relation of \eqref{a} is clear.
For the second we use induction on $m$.
The claim is clear for $0 \le m \le 1$.
For $m = 2$ we have $a^2b =b$ by Lemma \ref{cv}.
If $m$ is odd (resp.~even) then $m{-}1$ is even (resp.~odd) and
hence $a^{m} \cdot b = a (a^{m-1} b)= a b$ and
$a (a^{m-1} b)= a (a b) = a^2 b = b$
by induction.
Relation \eqref{b} is clear for $m{+}n < 3$.
We assume $m{+}n\geq 3$ and use induction on $m$.
For $m=1$, Lemma \ref{cv} gives
\[
a \cdot a^2 d = a^3 d =  a^4 + ad - a^2.
\]
For $m{+}n >3$, the inductive hypothesis implies
\begin{align*}
a^m \cdot a^n d
&=
a \cdot a^{m-1} a^n d
=
a \cdot \big( a^{m+n}+ a^{m+n-3} d - a^{m+n-2} \big)
\\[-1mm]
&=
a^{m+n+1} + a^{m+n-2} d - a^{m+n-1}.
\end{align*}
The proofs of the unnumbered relations are similar.
For relations \eqref{c} we have
\begin{align*}
ad \cdot a^2d
&=
adaad
=
aadad
=
aaadd
=
a^3 d^2
=
(a^4 +ad-a^2) d
\\[-1mm]
&=
a^4 d + a (cb+ad) - a^2 d
=
a (a^4 +ad - a^2)+ a(cb+ad)- a^2 d
\\[-1mm]
&=
a^5 + a^2 d - a^3 + acb = a^5 + a^2 d - a^3,
\\[-1mm]
a^2d \cdot  a^2d
&=
aadaad
=
aaadad
=
aaaadd
=
a^4  d^2
=
a (a^4 +ad -a^2)d
\\[-1mm]
&=
a^5d +a^2 d^2 - a^3d
=
a^5d + a^2 (cb + ad) -(a^4 +ad - a^2)
\\[-1mm]
&=
a^5d + a^3d - a^4 - ad + a^2
=
(a^2+1) a^3d - a^4 - ad + a^2
\\[-1mm]
&=
(a^2 +1) (a^4 + ad -a^2) - a^4 - ad  + a^2 =  a^6 + a^3d - a^4
\\[-1mm]
&=
a^6 + a^4 + ad -a^2 -a^4 = a^6 +  ad  - a^2,
\end{align*}
again using Lemma \ref{cv}.
\end{proof}

\begin{theorem}
We have the isomorphism $U(A^{CT}) \cong U(A^{C}) $.
\end{theorem}

\begin{proof}
The set $G^{CT}$ contains 128 elements; we put each generator in standard form
and self-reduce the set, obtaining the following 44 generators:
\[
\begin{array}{l@{\quad\;\;}l@{\quad\;\;}l@{\quad\;\;}l@{\quad\;\;}l}
aba {-} a^2b {+} b,&
 aca {-} a^2c,&
 acb {-} abc {+} a,&
 ada {-} a^2d,&
 adb {-} abd {+} b,
 \\
 adc {-} acd,&
 ba^2 {-} a^2b {+} b,&
 bab {-} ab^2,&
 bac {-} abc {+} a,&
 bad {-} abd {+} b,
 \\
  b^2a {-} ab^2,&
  bca {-} abc,&
   bcb {-} b^2c {-} b,&
 bda {-} abd {+} b,&
 bdb {-} bbd,
 \\
 bdc {-} bcd {-} a,&
 ca^2 {-} a^2c {-} c,&
  cab {-} abc {-} d {+} a,&
 cac {-} ac^2,&
 cad {-} acd,
 \\
  cba {-} abc {+} a,&
 cb^2 {-} b^2c,&
 cbc {-} bc^2 {-} c,&
 cbd {-} bcd {-} d,&
 c^2a {-} ac^2,
 \\
 c^2b {-} bc^2,&
 cda {-} acd,&
 cdb {-} bcd,&
 cdc {-} c^2d,&
  da^2 {-} a^2d,
 \\
  dab {-} abd {+} b,&
   dac {-} acd,&
 dad {-} ad^2,&
 dba {-} abd {+} b,&
 db^2 {-} b^2d,
\\
 dbc {-} bcd,&
 dbd {-} bd^2 {+} b,&
  dca {-} acd {-} c,&
 dcb {-} bcd {-} d,&
 dc^2 {-} c^2d,
 \\
 dcd {-} cd^2,&
 d^2a {-} ad^2,&
   d^2b {-} bd^2 {+} b,&
 d^2c {-} cd^2 {-} c.
 \end{array}
\]
There are 133 nontrivial compositions (omitted);
we include their normal forms, self-reduce, and obtain 16 generators,
coinciding with the Gr\"obner basis \eqref{UACGB}.
\end{proof}

\begin{conjecture}
If $C$ is a finite dimensional comtrans algebra (special or not)
then $U(C)$ is also finite dimensional.
\end{conjecture}



\end{document}